\documentclass[final,12pt]{article}
\usepackage[utf8]{inputenc}
\usepackage{anysize} %
\marginsize{2.5cm}{2.5cm}{2cm}{2cm} %
\usepackage{fancyhdr} %
\usepackage{lipsum} %
\usepackage{color} %
\usepackage{bm} %
\usepackage{amsfonts,amsthm,amsmath,amssymb} %
\usepackage{latexsym,graphicx,epstopdf} %
\usepackage{etoolbox} %
\usepackage{hyperref} %
\usepackage[left,mathlines]{lineno} %
\renewcommand{\a}{{\bf a}}
\renewcommand{\b}{{\bf b}}
\renewcommand{\c}{{\bf c}}

\newcommand{\e}{{\bf e}}

\renewcommand{\r}{{\bf r}}

\renewcommand{\u}{{\bf u}}
\renewcommand{\v}{{\bf v}}
\newcommand{\w}{{\bf w}}
\newcommand{\x}{{\bf x}}

\newcommand{\zero}{{\bf 0}}

\renewcommand{\aa}{{\bf A}}

\newcommand{\ee}{{\bf E}}

\newcommand{\hh}{{\bf H}}

\renewcommand{\ll}{{\bf L}}
\newcommand{\mm}{{\bf M}}

\newcommand{\rr}{{\bf R}}

\newcommand{\ww}{{\bf W}}

\newcommand{\aaa}{\mathbb{A}}

\renewcommand{\lll}{\mathbb{L}}

\newcommand{\nnn}{\mathbb{N}}

\newcommand{\rrr}{\mathbb{R}}

\newcommand{\dddd}{\mathcal{D}}

\newcommand{\btau}{{\bm\tau}}

\newcommand{\bsigma}{{\bm\sigma}}
\newcommand{\bepsilon}{{\bm\epsilon}}

\newcommand{\tr}{{\rm tr}}
\newcommand{\transpose}{\texttt{t}}

\newcommand{\seba}{Sebasti\'an A. Dom\'inguez-Rivera}
\newcommand{\sebaorgs}{Department of Mathematics and Statistics, University of Saskatchewan}
\newcommand{\sebacountry}{Saskatoon, Canada}
\newcommand{\sebalocation}{\sebaorgs, \sebacountry.}
\newcommand{\sebaemail}{s.dominguez@usask.ca}

\newcommand{\gonzalo}{Gonzalo A. Benavides}
\newcommand{\gonzaorgs}{Department of Mathematics, University of Maryland}
\newcommand{\gonzacountry}{College Park, USA}
\newcommand{\gonzalocation}{\gonzaorgs, \gonzacountry.}

\newcommand{\etallastname}{Benavides and Dom\'inguez-Rivera}
\newcommand{\funding}{This work was partially supported by the Pacific Institute for the Mathematical Sciences (PIMS).}

\newtheorem{theorem}{Theorem}

\newtheorem{remark}{Remark}

\title{Korn's inequality in anisotropic Sobolev spaces\thanks{\funding}}

\author{{\gonzalo\thanks{\gonzalocation} 
\quad \seba\thanks{\sebalocation}{\,\,\,\thanks{Corresponding author: 
\href{mailto:\sebaemail}{\sebaemail}.}} 
}}

\begin{document}

\maketitle

\begin{abstract}
Korn's inequality has been at the heart of much exciting research since its first appearance in the beginning of the 20th century. Many are the applications of this inequality to the analysis and construction of discretizations of a large variety of problems in continuum mechanics. In this paper, we prove that the classical Korn inequality holds true in {\it anisotropic Sobolev spaces}. We also prove that an extension of Korn's inequality, involving non-linear continuous maps, is valid in such spaces. Finally, we point out that another classical inequality, Poincare's inequality, also holds in anisotropic Sobolev spaces.

\begin{center}
  {\it This paper is dedicated to prof. Nilima Nigam on the occasion of her 50th birthday.}
\end{center}
\end{abstract}

\vspace{.25cm}


\section{Introduction}\label{section:intro}
Many linear and non-linear problems in continuum mechanics are of great interest in mathematics, engineering, and physics~\cite{ref:ciarlet2021}. As part of the analysis of such problems, Korn's inequality~\cite{ref:korn1906,ref:korn1908,ref:korn1909} has been intensively investigated, extended, and used to establish their well-posedness; see, e.g.~\cite{ref:horgan1995,ref:gatica2014}. It has particularly become an important part in the investigation, both theoretically and numerically speaking, of problems in linear elasticity, fluid-solid interaction, Newtonian fluids, fluids interface problems, eigenproblems in continuum mechanics, among others; see, for instance~\cite{ref:acosta2017,ref:hsiao2020,ref:gavioli2021,ref:dominguez2021,ref:dominguez2022}. 

The classical Korn inequality, named after A. Korn who first proved it in the beginning of the 1900s~\cite{ref:korn1906,ref:korn1908,ref:korn1909}, reads, in the usual notation for Sobolev spaces, as follows:
\begin{align}\label{eq:1stkorn}
  \|\u\|_{1,2,\Omega}\leq\,C\big(\|\bepsilon(\u)\|_{0,2,\Omega}+\|\u\|_{0,2,\Omega}\big),\quad\forall\,\u\in\hh^1(\Omega),
\end{align}
where $C>0$ is some positive constant, $\Omega$ is an open and bounded domain, and $\hh^1(\Omega)$ denotes the usual Hilbert space $H^1$ for vector fields defined on $\Omega$. The term $\bepsilon$ stands for the usual strain stress tensor, defined as the symmetric part of the gradient of a vector field in $\hh^1(\Omega)$. One can also obtain a sharper inequality as follows
\begin{align}\label{eq:2ndkorn}
  \inf_{\r\in\ker(\bepsilon)}\|\u-\r\|_{1,2,\Omega}\leq\, C\|\bepsilon(\u)\|_{0,2,\Omega},\quad\forall\,\u\in\hh^1(\Omega),
\end{align}
where $\ker(\bepsilon)=\{\v\in\hh^1(\Omega):\bepsilon(\v)=\zero\}$ denotes the kernel of the strain tensor. Many versions and extensions of these inequalities have been proven over the last decades~\cite{ref:friedrichs1947,ref:ting1972,ref:nitsche1981,ref:wang2003,ref:bauer2016-1,ref:ciarlet2005,ref:Kondratev1988,ref:dain2006}. Some recent work regarding Korn's inequalities has extended this inequality to vector fields belonging to the Sobolev space $\ww^{1,p}(\Omega)$ for $p\in(1,\infty)$, has studied it on general domains such as John domains and $(\epsilon,\delta)$-domains~\cite{ref:duran2004,ref:acosta2006,ref:duran2006,ref:acosta2016,ref:jiang2015,ref:jiang2017,ref:Friedrich2018,ref:lopez2018a,ref:lopez2018b}. There has also been development in proving that Korn's inequality in \eqref{eq:1stkorn} is valid even when the $\ll^2$-norm of the vector fields in $\hh^1(\Omega)$ is replaced by more general continuous mappings~\cite{ref:brenner2003,ref:ciarlet2015,ref:damlamian2018,ref:dominguez2021,ref:chipot2021}. 
Others have studied the necessary conditions needed for Korn's inequality in \eqref{eq:2ndkorn} may hold true. In this case we see that the essential part is to restrict the space $\hh^1(\Omega)$ to subspaces in which the kernel of the strain stress tensor is excluded. In particular, if, for example, tangential or normal components of the vector fields are zero on the boundary of the domain, then certain domains still support infinitesimal rigid motions; see, e.g.~\cite{ref:desvillettes2002,ref:bauer2016-2,ref:dominguez2022}. There has also been motivation to prove that Korn's inequality extends to incompatible tensor fields~\cite{ref:neff2002,ref:neff2011,ref:neff2012,ref:lewintan2021a,ref:lewintan2021b,ref:lewintan2021c}. On the other hand, Korn's inequality may fail in certain situations; for example, we note that Korn's inequality, at least in some of its forms, fails for vector fields in $\ww^{1,1}(\Omega)$; see e.g.~\cite{ref:ornstein1962,ref:conti2005}. See~\cite{ref:Neff2014} for some counterexamples regarding some some types of generalized Korn's inequality.

In this paper we prove that the classical Korn inequalities in \eqref{eq:1stkorn} and \eqref{eq:2ndkorn} hold true in more general Sobolev spaces than the classical $\ww^{1,p}(\Omega)$ spaces.
In fact, we prove that the Korn inequalities in \eqref{eq:1stkorn} and \eqref{eq:2ndkorn} are valid in {\it anisotropic Sobolev spaces}, which we denote by $\ww^{1,p,q}(\Omega)$. In such spaces, the vector fields considered belong to $\ll^p(\Omega)$ while their partial derivatives, in all possible directions, belong to $L^q(\Omega)$, with $1\leq p,q\leq\infty$. With such a definition, we note that $\ww^{1,p,p}(\Omega)$ reduces to the usual Sobolev spaces $\ww^{1,p}(\Omega)$. These more general Sobolev spaces were introduced (in an even more general way) and extensively studied in the works of R\'akosn\'ik~\cite{Rakosnik1979,Rakosnik1981}. More recently, Han utilized them as the appropriate setting for the study of existence of positive solutions of differential equations, see for instance~\cite{Han2020,han2022}. Additionally, we extend the contributions in~\cite{ref:gatica2007,ref:damlamian2018,ref:dominguez2021}, regarding a non-linear version of Korn's inequality, to these anisotropic Sobolev spaces. As a key factor in our proofs, we use that anisotropic Sobolev spaces are compactly embedded into the space $\ll^p(\Omega)$, for some values of $p$ and $q$; see, e.g.~\cite{Han2020}. Furthermore, as a direct consequence of the results mentioned above, we show that a generalized version of Poincare's inequality is also valid in these anisotropic Sobolev spaces.

The rest of the paper is organized as follows: in \autoref{section:anisotropic} we introduce the notation to be used throughout this paper and provide an overview of anisotropic Sobolev spaces. In \autoref{section:korns}, we state and prove that Korn's inequality, and an extension of it, holds in anisotropic Sobolev spaces. Finally, we draw some conclusions and point out some future work in \autoref{section:conclusion}

\section{Sobolev spaces}\label{section:anisotropic}
In this section we introduce some notation and give a brief overview of anisotropic Sobolev spaces, including the results needed for our study.

\subsection{Notation}\label{subsection:notation}
We begin this section by introducing the notation to be used throughout this paper. Vector fields will be denoted with lower-case bold alphabet letters whereas tensor fields are denoted with bold Greek letters. When using matrices, these will be denoted with bold upper-case alphabet letters. Between vector fields, the operation $\a\cdot\b$ is the 
standard dot product with induced norm $\|\cdot\|$. In turn, for tensors $\bsigma$ and $\btau$, the double dot product $\bsigma:\btau = \tr(\btau^\transpose\bsigma)$ is 
the usual inner product for matrices, which induces the Frobenius norm $\|\cdot\|_F$. 

For an open domain $\Omega$ of $\rrr^d$, $d\in\nnn$, the space $\dddd(\Omega)$ denotes all $C^{\infty}$ functions with compact support in $\Omega$, and $\dddd'(\Omega)$ denotes the space of distributions over $\Omega$. The space $W^{s,p}(\Omega)$ denotes the usual Sobolev space of scalar fields, for $s\in\rrr$ and $1\leq p\leq\infty$, with norm $\|\cdot\|_{s,p,\Omega}$. For vector fields, we use the notation $\ww^{s,p}(\Omega)$ with the corresponding norm simply denoted by $\|\cdot\|_{s,p,\Omega}$. If $p=2$, the Sobolev space $W^{s,2}(\Omega)$ is the usual Hilbert space $H^s(\Omega)$~\cite{ref:mclean2000}. The same notation is used for the Hilbert space for vector fields $\hh^s(\Omega):= \ww^{s,2}(\Omega)$. We use the convention $W^{0,p}(\Omega) = L^p(\Omega)$ and $\ww^{0,p}(\Omega) = 
\ll^p(\Omega)$, for $1\leq p\leq\infty$. Additionally, the space $\lll^p(\Omega)$ is the space of all tensor fields with entries in $L^p(\Omega)$. The norm of $L^p(\Omega)$, $\ll^p(\Omega)$, and $\lll^p(\Omega)$ are all denoted in the same way, as $\|\cdot\|_{0,p,\Omega}$.

For differential operators, $\nabla$ denotes the usual gradient operator acting on either a scalar field or a vector field. Finally, we employ $\zero$ to denote the zero vector, tensor, or operator, depending on the context.

\subsection{Anisotropic Sobolev spaces}
The content of this section corresponds to an overview of the most important results (important for this study) obtained in~\cite{Amrouche1994,Rakosnik1979,Han2020,Clarkson1936}. Let $\Omega$ be an open subset of $\rrr^d$.
For $1 \leq p \leq \infty$, a function $v \in L^p(\Omega)$ satisfies
\begin{equation}\label{duality-for-Lp}
v \in W^{-1,p}(\Omega) \quad \text{and} \quad \nabla v \in \ww^{-1,p}(\Omega). 
\end{equation}
This is because $v$ can be identified with a distribution $T_v$, defined as
\begin{equation*}
|T_v(\varphi)| = \left| \int_\Omega v \varphi \right| \leq \|v\|_{0,p,\Omega} \|\varphi\|_{1,q,\Omega}
\end{equation*}
and
\begin{equation*}
|\partial_i T_v(\varphi)| = |-T_v(\partial_i \varphi)| = \left| \int_\Omega v \partial_i \varphi \right| \leq \|v\|_{0,p,\Omega} \|\varphi\|_{1,q,\Omega},\quad\forall\,1\leq i\leq d,
\end{equation*}
for each $\varphi \in \mathcal{D}(\Omega)$, where $q \in [1,+\infty]$ is such that $\frac{1}{p} + \frac{1}{q} = 1$.
Here $\partial_i(\cdot) = \frac{\partial(\cdot)}{\partial x_i}$, for $1\leq i\leq d$.

We can now formally define an anisotropic Sobolev space.
Let $1 \leq p,q \leq +\infty$. We introduce the \textit{anisotropic} Sobolev space $W^{1,p,q}(\Omega)$ as
\begin{equation}\label{anisotropic-sobolev-space}
W^{1,p,q}(\Omega) := \big\{v \in L^p(\Omega): \nabla v \in \ll^q(\Omega)\big\}.
\end{equation}
From~\cite[Section 3, Theorem 1]{Rakosnik1979} we know that if $p,q \in [1,+\infty)$ the space $W^{1,p,q}(\Omega)$ is a separable Banach space with respect to the norm
\begin{equation}\label{anisotropic-norm}
\|v\|_{1,p,q,\Omega} := \|v\|_{0,p,\Omega} + \|\nabla v\|_{0,q,\Omega},
\end{equation}
and in addition, it is reflexive if $p,q \in (1,+\infty)$.

Analogously to the vector Sobolev spaces $\ww^{1,p}(\Omega)$, we can define the vector anisotropic Sobolev space $\ww^{1,p,q}(\Omega)$, endowed with the norm as defined as in \eqref{anisotropic-norm} but applied to vector fields. The norm for vector fields in $\ww^{1,p,q}(\Omega)$ will be simply denoted as for the scalar case, that is $\|\cdot\|_{1,p,q,\Omega}$.

For $1 \leq q < d$, we let $q^* := \frac{dq}{d-q}$ denote the Sobolev critical index.
We have the following result from~\cite[Lemma 2.1]{Han2020}, which we have rewritten according to our notation.

\begin{theorem}\label{lem:anisotropic-injections}
Let $\Omega$ be an open and bounded domain with a compact Lipschitz boundary.
Assume $1 \leq p \leq \infty$.
When $1 \leq q < d$, then the embedding $\iota: W^{1,p,q}(\Omega) \rightarrow L^s(\Omega)$ is continuous if $1 \leq s \leq \max\{q^*,p\}$ and compact if $1 \leq s < \max\{q^*,p\}$.
When $q=d$, then the embedding $\iota: W^{1,p,q}(\Omega) \hookrightarrow L^s(\Omega)$ is compact if $1 \leq s < \infty$.
When $d < q \leq \infty$, then the embedding $\iota: W^{1,p,q}(\Omega) \rightarrow L^s(\Omega)$ is continuous if $1 \leq s \leq \infty$ and compact if $1 \leq s < \infty$.
\end{theorem}

\section{Korn's inequality}\label{section:korns}
In this section we prove that Korn's inequality holds in anisotropic Sobolev spaces and show that an extension of it is also valid in these spaces.

\subsection{Extension to anisotropic Sobolev spaces}
The first result we present in this section extends the classical Korn inequality in \eqref{eq:1stkorn} to vector fields in $\ww^{1,p,q}(\Omega)$.
To that end we recall the following result from~\cite[Prop.~2.10(i)]{Amrouche1994}.
\begin{theorem}\label{thm:grad-regularity}
Let $\Omega$ be an open and bounded domain with Lipschitz boundary of $\rrr^d$. Aso let $m$ be any integer and $r$ be any real number with $1<r<\infty$.
If $p \in \dddd'(\Omega)$ is a distribution such that $\partial_i p \in W^{m-1,r}(\Omega)$, $1 \leq i \leq d$, then $p \in W^{m,r}(\Omega)$.
\end{theorem}

\begin{theorem}[Korn's inequality in $\ww^{1,p,q}(\Omega)$]\label{thm:korn-anisotropic}
Let $\Omega$ be a bounded domain of $\rrr^d$ with a compact and Lipschitz continuous boundary. 
Let $1\leq p,q \leq \infty$, $q \neq 1$.
Then, there exists a constant $C = C(\Omega,p,q)$ such that for all $\v \in \ww^{1,p,q}(\Omega)$
\begin{equation}\label{eq:1stkorn-anisotropic}
\|\v\|_{1,p,q,\Omega} \leq C \left\{ \|\v\|_{0,p,\Omega} + \|\bepsilon(\v)\|_{0,q,\Omega} \right\}
\end{equation}
\end{theorem}
\begin{proof}
The proof is a generalization of the one of Theorem 2.1 of~\cite{Ciarlet2010}, which in turn follows the one of Theorem 3.3 in Chapter 3 of~\cite{Duvaut1972}.

Let us define the space
\begin{equation}
\ee^{p,q}(\Omega) := \{ \v \in \ll^p(\Omega): \bepsilon(\v) \in \lll_s^q(\Omega)\},
\end{equation}
where $\lll^q_s(\Omega)$ is the space of symmetric tensor fields in $\lll^q(\Omega)$. Then, $\ee^{p,q}(\Omega)$ endowed with the norm $\|\v\| := \|\v\|_{0,p,\Omega} + \|\bepsilon(\v)\|_{0,q,\Omega}$ is a Banach space. Indeed, the condition $\bepsilon(\v) \in \lll_s^q(\Omega)$ defining $\ee^{p,q}(\Omega)$ is understood in the sense of distributions, that is, it means that there exist symmetrical functions $e_{ij}(\v)$ in $L^q(\Omega)$, such that
\begin{equation*}
\int_\Omega e_{ij}(\v) \varphi = -\frac{1}{2} \int_\Omega \left( v_i \partial_j \varphi + v_j \partial_i \varphi \right), \quad \forall \varphi \in \mathcal{D}(\Omega).
\end{equation*}
Now, let $\{\v^k\}_{k\in\nnn}$ be a Cauchy sequence with respect to $\|\cdot\|$ of elements $\v^k = (v^k_1,\ldots,v^k_d) \in \ee^{p,q}(\Omega)$.
The definition of $\|\cdot\|$ implies that for each $i,j=1,\dots,d$, the sequences $\{v^k_i\}_{k\in\nnn}$ and $\{e_{ij}(\v^k)\}_{k\in\nnn}$ are Cauchy in $L^p(\Omega)$ and $L^q(\Omega)$, respectively. Being these two spaces complete, there exist functions $v_i \in L^p(\Omega)$ and $e_{ij} \in L^q(\Omega)$ such that
\begin{equation*}
\|v^k_i - v_i\|_{0,p,\Omega} \xrightarrow{k\to\infty} 0 \quad \text{and} \quad \|e_{ij}(\v^k) - e_{ij}\|_{0,q,\Omega}\xrightarrow{k\to\infty} 0.
\end{equation*}
Moreover, since $e_{ij}(\v^k) = e_{ji}(\v^k)$ for each natural $k$, it is clear that $e_{ij} = e_{ji}$.

Given a function $\varphi \in \mathcal{D}(\Omega)$, we let $k \to \infty$ in the relations
\begin{equation*}
\int_\Omega e_{ij}(\v^k) \varphi = -\frac{1}{2} \int_\Omega \left( v^k_i \partial_j \varphi + v^k_j \partial_i \varphi \right), \quad \forall \varphi \in \mathcal{D}(\Omega).
\end{equation*}
to deduce $e_{ij} = e_{ij}(\v)$.
Thus $\v \in \ee^{p,q}(\Omega)$, and $(\v^k)_{k\in\nnn}$ converges to $\v$ in $\|\cdot\|$.

Furthermore, $\ee^{p,q}(\Omega)$ and $\ww^{1,p,q}(\Omega)$ coincide.
Indeed, since $\|\bepsilon(\cdot)\|_{0,q,\Omega}\leq c\|\nabla(\cdot)\|_{0,q,\Omega}$ for some positive constant $c$, we clearly have that $\ww^{1,p,q}(\Omega) \subseteq \ee^{p,q}(\Omega)$.
For the reverse inclusion, consider $\v = (v_1,\ldots,v_d) \in \ee^{p,q}(\Omega)$.
Then for $1 \leq i,j,k \leq d$, duality arguments \eqref{duality-for-Lp} give
\begin{equation*}
\partial_k v_i \in W^{-1,p}(\Omega) \subseteq \dddd'(\Omega), \quad \partial_j(\partial_k v_i) = \{ \partial_j e_{ik}(\v) + \partial_k \e_{ij}(\v) - \partial_i e_{jk}(\v) \} \in W^{-1,q}(\Omega).
\end{equation*}
Hence, by \autoref{thm:grad-regularity} with $m=0$ we know that $\partial_k v_i \in L^q(\Omega)$, and therefore $\v \in \ww^{1,p,q}(\Omega)$.

The identity mapping $\iota$ from $\ww^{1,p,q}(\Omega)$ equipped with $\|\cdot\|_{1,p,q,\Omega}$ into $\ee^{p,q}(\Omega)$ equipped with $\|\cdot\|$ is bounded (clearly $\|\cdot\| \leq c \|\cdot\|_{1,p,q,\Omega}$, for some positive constant $c$ and surjective as $\ee^{p,q}(\Omega) = \ww^{1,p,q}(\Omega)$.
In this way, the bounded inverse Theorem~\cite[Corollary 2.7]{Brezis2011} then shows that the inverse mapping $\iota^{-1}$ is also bounded, which is exactly what we were aiming to prove.
\end{proof}
\begin{remark}
We stress the fact that the previous result, unlike the forthcoming \autoref{thm:korn-quotientspace} and \autoref{thm:nonlinearkorn}, holds regardless of the embedding $\iota: W^{1,p,q}(\Omega) \hookrightarrow L^q(\Omega)$ being compact or not.
\end{remark}

Now, let $\aaa^d$ denote the space of all real $d \times d$ skew-symmetric matrices. We define the space of {\it rigid motions} of $\Omega$, denoted by $\rr\mm(\Omega)$, as follows
\begin{equation}
    \rr\mm(\Omega) := \big\{ \v \in \ww^{1,p,q}(\Omega): \v(\x) = \aa \x + \c \quad \text{a.e. $\x\in\Omega$},\quad \aa \in \aaa^d,\,\, \c \in \rrr^d \big\}.
\end{equation}
The following result characterizes the kernel of the strain stress tensor in $\ww^{1,p,q}(\Omega)$.
\begin{theorem}\label{thm:kernel-symgrad}
Let $\Omega$ be an open, and connected domain in $\rrr^d$. Then, the kernel of the strain stress tensor, $\ker(\bepsilon) := \{\v \in \ww^{1,p,q}(\Omega):\,\bepsilon(\v) = \mathbf{0}\,\,\text{in}\,\,\Omega\}$, coincides with the space of rigid motions, that is
\begin{equation}
\ker(\bepsilon)=\rr\mm(\Omega).
\end{equation}
In addition, the kernel of the strain stress tensor is finite-dimensional, with dimension $\frac{d(d+1)}{2}$.
\end{theorem}
\begin{proof}
The proof is a generalization of the one of~\cite[Theorem 2.2]{Ciarlet2010}.
Note that the inclusion $\rr\mm(\Omega)\subseteq\ker(\bepsilon)$ is straightforward.

Observe that for any $\v \in \ww^{1,p,q}(\Omega)$ and $i,j,k=1,\dotsc,d$, it holds
\begin{equation*}
\int_\Omega (\partial_j v_i) \partial_k \varphi
= \int_\Omega \left( e_{ij}(\v) \partial_k \varphi + e_{ik}(\v)\partial_j \varphi - e_{jk}(\v)\partial_i \varphi \right) \quad \forall \varphi \in \mathcal{D}(\Omega),
\end{equation*}
because both sides of the equality coincide with $\displaystyle-\int_\Omega v_i \partial_{kj} \varphi$.
Therefore, if $\v \in \ww^{1,p,q}(\Omega)$ is such that $\bepsilon(\v) = \zero$, then
\begin{equation*}
\partial_k T_{\partial_j v_i}(\varphi) = - \int_\Omega (\partial_j v_i)\partial_k\varphi  = 0, \quad \forall \varphi \in \mathcal{D}(\Omega), \quad \forall j,k = 1,\dotsc,d.
\end{equation*}
Thus, given that $\Omega$ is connected, we can use~\cite[Corollary 3.1.6]{Hormander2003} in combination with~\cite[Theorem 13.17]{Fitzpatrick2009} to deduce the existence of constants $a_{ij}$, $i,j=1,\dotsc,d$, such that $\partial_j v_i = a_{ij}$ a.e.~in $\Omega$.
Moreover, the condition $e_{ij}(\v) = 0$ implies that $a_{ij} = -a_{ji}$.

Let $w_i(\x) := \sum_{j=1}^d a_{ij} x_j$ for $\x \in \Omega$, $i=1,\dotsc,d$.
Then for each $\varphi \in \mathcal{D}(\Omega)$
\begin{equation*}
\int_\Omega v_i \partial_j \varphi = - \int_\Omega (\partial_j v_i)\varphi = - a_{ij} \int_\Omega \varphi = -\int_\Omega (\partial_j w_i) \varphi = \int_\Omega w_i \partial_j \varphi,
\end{equation*}
that is $\partial_j T_{v_i-w_i} = 0$ for each $j=1,\dotsc,d$.
In this way, using~\cite[Corollary 3.1.6]{Hormander2003} and~\cite[Theorem 13.17]{Fitzpatrick2009} there exist constants $c_i$ such that $v_i(\x) = w_i(\x) + c_i$ a.e.~in $\Omega$.
Thus, we have shown that $\v(\x) = \aa \x + \c$ a.e.~in $\Omega$ with $\aa := (a_{ij})$ and $\c := (c_i)$.
\end{proof}

\begin{theorem}[Korn's inequality in the quotient space $\ww^{1,p,q}(\Omega)/\ker(\bepsilon)$]\label{thm:korn-quotientspace}
Let $\Omega$ be a bounded domain of $\rrr^d$ with compact, Lipschitz boundary such that $\ker(\bepsilon)$ is finite-dimensional.
Define the quotient space
\begin{equation}
\dot \ww^{1,p,q}(\Omega) := \ww^{1,p,q}(\Omega)/\ker(\bepsilon),
\end{equation}
equipped with the norm
\begin{equation*}
\|\dot \v\|_{1,p,q,\Omega} := \inf_{\r \in \ker(\bepsilon)} \|\v - \r\|_{1,p,q,\Omega} \quad \text{for all} \quad \dot \v \in \dot \ww^{1,p,q}(\Omega),
\end{equation*}
the space $\dot \ww^{1,p,q}(\Omega)$ is a Banach space.
Then:
\begin{enumerate}
\item\label{it:quotientspace} If $1 \leq p,q\leq \infty$, $q \neq 1$, are such that the embedding $\iota: W^{1,p,q}(\Omega) \hookrightarrow L^p(\Omega)$ is compact (see Lemma \ref{lem:anisotropic-injections}), then there exists a constant $\dot C := \dot C(\Omega,p,q)$ such that Korn's inequality in $\dot \ww^{1,p,q}(\Omega)$ holds, that is,
\begin{equation}\label{eq:2ndkorn-anisotropic}
\|\dot \v\|_{1,p,q,\Omega} \leq \dot C \|\bepsilon(\dot \v)\|_{0,q,\Omega} \quad \forall \dot \v \in \dot \ww^{1,p,q}(\Omega),
\end{equation}
where $\bepsilon(\dot \v) := \bepsilon(\w)$ for any $\w \in \dot \v$.

\item\label{it:quotient-to-normal} Conversely, Korn's inequality in $\dot \ww^{1,p,q}(\Omega)$ implies Korn's inequality in $\ww^{1,p,q}(\Omega)$ (see \autoref{thm:korn-anisotropic}).
\end{enumerate}
\end{theorem}
\begin{proof}
\begin{enumerate}
\item
Let $M$ be the dimension of $\ker(\bepsilon)$. 
A corollary of Hahn-Banach theorem in normed vector spaces~\cite{Brezis2011} ensures the existence of $M$ linear and bounded functionals $\ell_\alpha$ on $\ww^{1,p,q}(\Omega)$, $1 \leq \alpha \leq M$, such that: given an element $\r \in \ker(\bepsilon)$, $\r$ is the null vector if and only if $\ell_\alpha(\r) = 0$, $1 \leq \alpha \leq M$.
Indeed, if $\{r_1,\dotsc,r_M\}$ is a basis of $\ker(\bepsilon)$, the canonical basis of $\ker(\bepsilon)'$ (the dual space of $\ker(\bepsilon)$ with respect to the $\|\cdot\|_{1,p,q,\Omega}$) is $\{f_1,\dotsc,f_M\}$, which satisfies $f_\alpha(r_\beta) = \delta_{\alpha\beta}$ for each $1 \leq \alpha,\beta \leq M$.
We extend, via Hahn-Banach Theorem~\cite[Corollary 1.2]{Brezis2011} each $f_\alpha$ to a linear and bounded functional $\ell_\alpha$ on $\ww^{1,p,q}(\Omega)$.

We claim the existence of a positive constant $D = D(p,q,\Omega)$, such that
\begin{equation}\label{bounding-by-HB}
\|\v\|_{1,p,q,\Omega} \leq D \left\{ \|\bepsilon(\v)\|_{0,q,\Omega} + \sum_{\alpha=1}^M |\ell_\alpha(\v)| \right\} \quad \forall \v \in \ww^{1,p,q}(\Omega).
\end{equation}
Indeed, let's assume the contrary.
Then, there exists a sequence $\{\v^k\}_{k\in\nnn}$ on $\ww^{1,p,q}(\Omega)$, such that $\|\v^k\|_{1,p,q,\Omega} = 1$ for all $k \in \nnn$
\begin{equation}\label{functional-norm}
\|\bepsilon(\v^k)\|_{0,q,\Omega} + \sum_{\alpha=1}^M |\ell_\alpha(\v^k)| \xrightarrow{k\to\infty} 0.
\end{equation}
Being $\{\v^k\}_{k\in\nnn}$ bounded in $\ww^{1,p,q}(\Omega)$ by compactness of the embedding $\iota$ there exists a subsequence $\{\v^l\}_{l\in\nnn}$ that converges in $L^p(\Omega)$.
Since the sequence $\{\bepsilon(\v^l)\}_{l\in\nnn}$ is convergent in $\lll_s^q$, the subsequence $\{\v^l\}_{l\in\nnn}$ is a Cauchy sequence with respect to the norm $\|\cdot\| := \|\cdot\|_{0,p,\Omega} + \|\bepsilon(\cdot)\|_{0,q,\Omega}$, hence also with respect to the norm $\|\cdot\|_{1,p,q,\Omega}$ by Korn's inequality in $\ww^{1,p,q}(\Omega)$ (see \autoref{thm:korn-anisotropic}).
As a consequence, there exists $\v \in \ww^{1,p,q}(\Omega)$ such that 
\begin{equation*}
\|\v^l-\v\|_{1,p,q,\Omega} \xrightarrow{l\to\infty} 0.
\end{equation*}
But, by continuity, it follows from \eqref{functional-norm} that $\bepsilon(\v) = \zero$ (that is, $\v \in \ker(\bepsilon)$) and $\ell_\alpha(\v) = 0$, $1 \leq \alpha \leq M$, and therefore $\v = \bf{0}$; in contradiction with the fact that $\|\v^l\|_{1,p,q,\Omega} = 1$ for all $l \in \nnn$.
Thus, \eqref{bounding-by-HB} holds true.

In this way, given any $\v \in \ww^{1,p,q}(\Omega)$, we can choose $\r(\v) \in \ker(\bepsilon)$ such that $\ell_\alpha(\v+\r(\v)) = 0$, for each $1 \leq \alpha \leq M$ (it is equivalent to solve a uniquely determined $M \times M$ system of equations).
Then, by \eqref{bounding-by-HB}
\begin{equation*}
\|\dot \v\|_{1,p,q,\Omega} = \inf_{\r \in \ker(\bepsilon)} \|\v + \r\|_{1,p,q,\Omega} \leq \|\v + \r(\v)\|_{1,p,q,\Omega} \leq D\|\bepsilon(\v)\|_{0,q,\Omega} = D \|\bepsilon(\dot \v)\|_{0,q,\Omega},
\end{equation*}
which completes the proof.

\item By means of contradiction, assume that there exists a sequence $\{\v^k\}_{k\in\nnn}$ in $\ww^{1,p,q}(\Omega)$ such that $\|\v^k\|_{1,p,q,\Omega} = 1$ for each $k \in \nnn$ and
\begin{equation}\label{quotient-to-normal-contradictionstatement}
\|\v^k\|_{0,p,\Omega} + \|\bepsilon(\v^k)\|_{0,q,\Omega} \xrightarrow{k\to\infty} 0.
\end{equation}
By definition of the norm $\|\cdot\|_{1,p,q,\Omega}$ in $\dot \ww^{1,p,q}(\Omega)$, there exists a sequence $\{\r^k\}_{k\in\nnn}$ in $\ker(\bepsilon)$ that satisfies
\begin{align*}
\|\v^k-\r^k\|_{1,p,q,\Omega}
&< \|\dot \v^k\|_{1,p,q,\Omega} + \frac{1}{k}\\
& \leq \|\v^k\|_{1,p,q,\Omega} + \frac{1}{k}
\leq 2.
\end{align*}
The space $\ker(\bepsilon)$ being finite-dimensional, the inequality $\|\r^k\|_{1,p,q,\Omega} \leq \|\v^k-\r^k\|_{1,p,q,\Omega} + \|\v^k\|_{1,p,q,\Omega} \leq 3$ for all $k \in \nnn$ implies the existence of a subsequence $\{\r^l\}_{l\in\nnn}$ that converges in $\ww^{1,p,q}(\Omega)$ to an element $\r \in \ker(\bepsilon)$.
Besides, by Korn's inequality in $\dot \ww^{1,p,q}(\Omega)$ and noting that $\v^l-\r^l \in \dot \v^l$,
\begin{multline*}
\|\v^l-\r^l\|_{1,p,q,\Omega} < \|\dot \v^l\|_{1,p,q,\Omega} + \frac{1}{l} \leq C \|\bepsilon(\dot \v^l)\|_{0,q,\Omega} + \frac{1}{l}\\
= C \|\bepsilon(\v^l-\r^l)\|_{0,q,\Omega} + \frac{1}{l}
= C \|\bepsilon(\v^l)\|_{0,q,\Omega} + \frac{1}{l}
\xrightarrow{l\to\infty} 0;
\end{multline*}
which readily implies that $\|\v^l-\r\|_{1,p,q,\Omega} \xrightarrow{l\to\infty} 0$.
Hence, $\|\v^l-\r\|_{0,p,\Omega} \xrightarrow{l\to\infty} 0$, which force $\r$ to be $\bf{0}$ since $\|\v^l\|_{0,p,\Omega} \xrightarrow{l\to\infty} 0$ (cf.~\eqref{quotient-to-normal-contradictionstatement}).
Thus, we reach to the conclusion that $\|\v^l\|_{1,p,q,\Omega} \xrightarrow{l\to\infty} 0$, which contradicts $\|\v^l\|_{1,p,q,\Omega} = 1$.
\end{enumerate}
\end{proof}

\subsection{Extension to continuous mappings}\label{subsection:extension}
Consider a continuous map $F:\ww^{1,p,q}(\Omega)\to\rrr$, different from the $\|\cdot\|_{1,p,q,\Omega}$ norm (cf.~\eqref{anisotropic-norm}). Let $N(F)$ be the set of zeros of $F$, that is
\begin{align*}
    N(F):=\Big\{\v\in \ww^{p,q}(\Omega):\,F(\v) = 0\Big\}.
\end{align*}
Note that the set $N(F)$ could be the empty set. We assume that this set satisfies any of the two conditions below
\begin{align}
 N(F)\cap \ker(\bepsilon) = \{\zero\}\quad\text{or}\quad N(F)\cap \ker(\bepsilon) = \emptyset.\label{eq:zerocap}
\end{align}
Korn's inequality in $\ww^{1,p,q}(\Omega)$ for general continuous maps is stated and proven below.

\begin{theorem}\label{thm:nonlinearkorn}
Assume $\Omega$ is a bounded domain of $\rrr^d$ with compact, Lipschitz boundary such that $\ker(\bepsilon)$ is finite-dimensional.
Let $p,q \in (1,+\infty)$ such that the embedding $\iota: W^{1,p,q}(\Omega) \hookrightarrow L^p(\Omega)$ is compact (see \autoref{lem:anisotropic-injections}).
Then, there is a constant $C>0$ such that
\begin{align}
\|\u\|_{1,p,q,\Omega} \leq C\Big(\|\bepsilon(\u)\|_{0,q,\Omega} + |F(\u)|\Big),\quad\forall\, \u \in 
\ww^{1,p,q}(\Omega).\label{eq:extendedphikorns}
\end{align}
\end{theorem}
\begin{proof}
By contradiction, suppose there is a sequence $\{\u^k\}_{k\in \nnn} \in \ww^{1,p,q}(\Omega)$ such that

\begin{align*}
\|\u^k\|_{1,p,q,\Omega} = 1,\quad \|\bepsilon(\u^k)\|_{0,q,\Omega} + |F(\u^k)| \xrightarrow{k\to\infty} 0
\end{align*}
Since $\{\u^k\}_{k\in\nnn}$ is a bounded sequence in the $\|\cdot\|_{1,p,q,\Omega}$ norm and $\ww^{1,p,q}(\Omega)$ is reflexive (for $1 < p,q < +\infty$), we know from~\cite[Theorem 3.18]{Brezis2011} that there is a subsequence $\{\u^{k_l}\}_{l\in\nnn}$ of $\{\u^k\}_{k\in\nnn}$ and $\u\in \ww^{1,p,q}(\Omega)$ such that $\u^{k_l} \xrightarrow{l\to\infty} \u$ weakly in $\ww^{1,p,q}(\Omega)$.
The embedding $\iota: W^{1,p,q}(\Omega) \hookrightarrow L^p(\Omega)$ being compact, it further implies that $\u^{k_l} \to \u$ strongly in $\ll^p(\Omega)$.

In turn, we see that $|F(\u^k)| \xrightarrow{k\to\infty} 0$ and $\|\bepsilon(\u^k)\|_{0,q,\Omega} \xrightarrow{k\to\infty} 0$.
Besides, from \autoref{thm:korn-anisotropic},
\begin{align*}
 \|\u^{k_l}-\u^{k_m}\|_{1,p,q,\Omega} \leq C\,\Big(\|\bepsilon(\u^{k_l})-\bepsilon(\u^{k_m})\|_{0,q,\Omega} + 
\|\u^{k_l}-\u^{k_m}\|_{0,p,\Omega}\Big) \quad \forall l,m \in \nnn
\end{align*}
From the observations above we have that $\{\u^{k_l}\}_{l\in\nnn}$ is a Cauchy sequence in $\ww^{1,p,q}(\Omega)$ and thus $\u^{k_l}\xrightarrow{l\to\infty} \u$ strongly in $\ww^{1,p,q}(\Omega)$.
The continuity of $F$ then gives $|F(\u^{k_l})| \xrightarrow{l\to\infty} |F(\u)|$ and therefore $\u\in N(F)$.
Also, the fact that $\|\bepsilon(\u^k)\|_{0,q,\Omega}\xrightarrow{k\to\infty} 0$ and the continuity of $\bepsilon$ in $\ww^{1,p,q}(\Omega)$ implies that $\u\in \ker(\bepsilon)$ that is, $\u$ belongs to $N(F)\cap \ker(\bepsilon)$.
That is, if the second condition of \eqref{eq:zerocap} holds, we have reached a contradiction.
Otherwise, the first condition of \eqref{eq:zerocap} would imply that $\u = \zero$, condition that crashes with $\u^{k_l}\to \u$ strongly in $\ww^{1,p,q}(\Omega)$, because from it it follows that $1 = \|\u^{k_l}\|_{1,p,q,\Omega}\to \|\u\|_{1,p,q,\Omega} = 1$.
\end{proof}
\begin{remark}
    Note that the functional $F$ is meant to be different from the $\|\cdot\|_{1,p,q,\Omega}$ norm as this case corresponds to the usual Korn's inequality in \autoref{thm:korn-anisotropic} and it is used in the proof of the result above.
\end{remark}

\subsection{A note on Poincare's inequality}\label{subsect:poincare}
The results already obtained for Korn's inequality in the spaces $\ww^{1,p,q} (\Omega)$ can also be proved, with no much more work, for the classical Poincare inequality. In fact, define
\begin{equation}
\ker(\nabla) := \{\v \in \ww^{1,p,q}(\Omega): \nabla \v = \mathbf{0} \quad \text{in} \quad \Omega\}.
\end{equation}
Using both~\cite[Corollary 3.1.6]{Hormander2003} and~\cite[Theorem 13.17]{Fitzpatrick2009} it readily follows that if $\Omega$ is an open connected set, then
\begin{equation*}
\ker(\nabla) = \big\{ \v \in \ww^{1,p,q}(\Omega): \v = \c \quad \text{a.e.~$\in\Omega$},\quad \c \in \rrr^d \big\},
\end{equation*}
and in particular, it is finite-dimensional with dimension equal to $d$.

\begin{theorem}[Poincare's inequality in the quotient space $\ww^{1,p,q}(\Omega)/\ker(\nabla)$]\label{thm:poincare-quotientspace}
Let $\Omega$ be a bounded domain of $\rrr^d$ with compact, Lipschitz boundary, such that $\ker(\nabla)$ is finite-dimensional.
Define the quotient space
\begin{equation}
\dot \ww^{1,p,q}(\Omega) := \ww^{1,p,q}(\Omega)/\ker(\nabla),
\end{equation}
equipped with the norm
\begin{equation*}
\|\dot \v\|_{1,p,q,\Omega} := \inf_{\r \in \ker(\nabla)} \|\v - \r\|_{1,p,q,\Omega}, \quad \forall\, \dot \v \in \dot \ww^{1,p,q}(\Omega),
\end{equation*}
the space $\dot \ww^{1,p,q}(\Omega)$ is a Banach space.
Then:
If $1 \leq p,q \leq +\infty$ are such that the embedding $\iota: W^{1,p,q}(\Omega) \hookrightarrow L^p(\Omega)$ is compact (see Lemma \ref{lem:anisotropic-injections}), then there exists a constant $\dot C := \dot C(\Omega,p,q)$ such that Poincare's inequality in $\dot \ww^{1,p,q}(\Omega)$ holds, that is,
\begin{equation*}
\|\dot \v\|_{1,p,q,\Omega} \leq \dot C \|\nabla(\dot \v)\|_{0,q,\Omega}, \quad \forall \dot \v \in \dot \ww^{1,p,q}(\Omega),
\end{equation*}
where $\nabla(\dot \v) := \nabla \w $ for any $\w \in \dot \v$.
\end{theorem}

\begin{proof}
The proof follows almost verbatim as the proof of \autoref{it:quotientspace} of \autoref{thm:korn-quotientspace} by replacing the role of $\bepsilon$ by $\nabla$, and without utilizing \autoref{thm:korn-anisotropic} and hence it is allowed for $q$ to take the value $1$.
\end{proof}

Consider a continuous map $F:\ww^{1,p,q}(\Omega)\to\rrr$, different from the $\|\cdot\|_{1,p,q,\Omega}$ norm (cf.~\eqref{anisotropic-norm}). Let $N(F)$ be the set of zeros of $F$, that is
\begin{align*}
    N(F):=\Big\{\v\in \ww^{p,q}(\Omega):\,F(\v) = 0\Big\}.
\end{align*}
Note that the set $N(F)$ could be the empty set. We assume that this set satisfies any of the two conditions below
\begin{align}
 N(F)\cap \ker(\nabla) = \{\zero\}\quad\text{or}\quad N(F)\cap \ker(\nabla) = \emptyset.
\end{align}

Poincare's inequality in $\ww^{1,p,q}(\Omega)$ for general continuous maps is stated and proven below.

\begin{theorem}\label{thm:poincare-nonlinear}
Assume $\Omega$ is a bounded domain of $\rrr^d$ with compact, Lipschitz boundary.
Let $1 < p,q < \infty$ be such that the embedding $\iota: W^{1,p,q}(\Omega) \hookrightarrow L^p(\Omega)$ is compact (see \autoref{lem:anisotropic-injections}).
Then, there is a constant $C>0$ such that
\begin{align}
\|\u\|_{1,p,q,\Omega} \leq C\Big(\|\nabla \u\|_{0,q,\Omega} + |F(\u)|\Big),\quad\forall\, \u \in 
\ww^{1,p,q}(\Omega).
\end{align}
\end{theorem}
\begin{proof}
Its proof follows the same ideas of the one of \autoref{thm:nonlinearkorn} by replacing the role of $\bepsilon$ by $\nabla$.
It does not require the use of \autoref{thm:korn-anisotropic}.
\end{proof}

\section{Conclusions}\label{section:conclusion}
In this paper we proved that the classical Korn inequality, usually posed for vector fields in $\hh^1(\Omega)$, is valid in the anisotropic Sobolev spaces $\ww^{1,p,q}(\Omega)$, under different assumptions on the Sobolev parameters $p$ and $q$. In particular, we point out that \autoref{thm:korn-anisotropic} and \autoref{thm:korn-quotientspace} imply that Korn's inequality (cf. \autoref{eq:1stkorn-anisotropic} and \autoref{eq:2ndkorn-anisotropic}) holds for vector fields in $\ww^{1,1,q}(\Omega)$ as long as $q$ is strictly larger than 1. In fact, Korn's inequality for vector fields in $\ww^{1,1,1}(\Omega)$ cannot hold as previously proved in~\cite{ref:ornstein1962,ref:conti2005}. This follows because $\ww^{1,1,1}(\Omega)$ reduces to the classical Sobolev space $\ww^{1,1}(\Omega)$.

Furthermore, we were also able to extend the results from~\cite{ref:gatica2007,ref:damlamian2018,ref:dominguez2021} and proved that a non-linear version of Korn's inequality (cf. \eqref{eq:extendedphikorns}) holds true in these anisotropic Sobolev spaces (cf. \autoref{thm:nonlinearkorn}).
Similarly, we also proved that Poincare's inequality, and a more general non-linear version of it, hold in this more general Sobolev spaces (cf. results in \autoref{subsect:poincare}). Note that to be able to extend Korn's and Poincare's inequality to a non-linear version, the continuous mappings considered as in \autoref{thm:nonlinearkorn} and \autoref{thm:poincare-nonlinear} cannot share any elements (or only share the zero vector) with the kernels of the strain stress tensor or the gradient, respectively.

In~\cite{Rakosnik1979,Rakosnik1981}, the author studied \textit{anisotropic Sobolev spaces} in a more general way, in which for example, each distributional derivative lies in a possibly different $L^p$-space.
A natural question question would be, then, whether the Korn inequality also holds for this more general setting. Nevertheless, it is not yet clear in what $L^p$-spaces we should measure each component of the symmetric part of the gradient, and hence an appropriate expression for a Korn-type inequality in this general setting is unknown to the authors. This is left as subject of future work.

\section*{Acknowledgments}
\seba\ thanks the support of the Pacific Institute for the Mathematical Sciences (PIMS) through a PIMS postdoctoral fellowship.
The authors are very grateful to prof.~J.~R\'{a}kosn\'{i}k (Institute of Mathematics CAS) for providing us with digital copies of \cite{Rakosnik1979,Rakosnik1981}, and to prof.~Gabriel N.~Gatica (Universidad de Concepci\'on) for his comments which helped improving this manuscript.

\paragraph{Note for the reader}
Please note that this manuscript represents a preprint only and has not been (or is in the process of being) peer-reviewed. A DOI link will be made available for this ArXiv preprint as soon as the peer-reviewed version is published online.

\bibliography{mathscinet-references.bib}

\end{document}